\documentclass[11pt,a4paper]{amsart}
\usepackage{amsmath,amsthm,amssymb}
\usepackage{latexsym}
\usepackage{graphicx}

\allowdisplaybreaks
\numberwithin{equation}{section}

\newtheorem{theorem}{Theorem}[section]

\newtheorem{lemma}{Lemma}[section]

\newtheorem{Corollary}{Corollary}[section]

\newtheorem{remark}{Remark}[section]

\newcommand{\nt}{\noindent}

\newcommand{\ovl}{\overline}

\newcommand{\ti}{\tilde}

\newcommand{\lt}{\left}
\newcommand{\rt}{\right}

\newcommand{\bn}{{\mathbb{N}}}
\newcommand{\br}{{\mathbb{R}}}

\newcommand{\bt}{{\mathbb{T}}}

\renewcommand{\a}{\alpha}

\newcommand{\s}{\sigma}
\renewcommand{\ss}{\Sigma}

\newcommand{\p}{\phi}
\newcommand{\pp}{\Phi}
\newcommand{\ps}{\Psi}

\renewcommand{\th}{\theta}
\renewcommand{\d}{\delta}

\newcommand{\g}{\gamma}
\newcommand{\ga}{\Gamma}
\newcommand{\e}{\epsilon}
\newcommand{\ve}{\varepsilon}

\renewcommand{\lll}{\lesssim}
\renewcommand{\ggg}{\gtrsim}

\newcommand{\ov}{{\ovl o}}
\newcommand{\wto}{{\stackrel{*}{\to}}}
\newcommand{\nk}{{N_k}}

\newcommand{\sz}{(S)}

\begin{document}

\title[On the growth of the polynomial entropy integrals \ldots]
{On the growth of the polynomial entropy \\integrals for  measures
in the Szeg\H{o} class}

\author[S. Denisov, S. Kupin]{S. Denisov, S. Kupin}

\address{Mathematics Department, University of Wisconsin--Madison,   480 Lincoln Dr.,  Madison, WI 53706 USA}
\email{denissov@math.wisc.edu}

\address{Universit\'e Bordeaux, IMB, UMR 5251, F-33405 Talence Cedex France \newline\indent
CNRS, IMB, UMR 5251, F-33405 Talence, France.}
\email{skupin@math.u-bordeaux1.fr}


\thanks{The first author is partially supported by NSF Grants DMS-1067413 and DMS-0758239.
The second author is partially supported by grants  ANR-07-BLAN-024701 and  ANR-09-BLAN-005801}


\begin{abstract}
Let $\s$ be a probability Borel measure on the unit circle $\bt$
and $\{\p_n\}$ be the orthonormal polynomials with respect to $\s$.
We say that $\s$ is a Szeg\H o measure, if it has an arbitrary
singular part $\s_s$, and $\int_\bt \log \s' dm>-\infty$, where
$\s'$ is the density of the absolutely continuous part of $\s$, $m$
being the normalized Lebesgue measure on $\bt$.  The entropy
integrals for $\p_n$ are defined as
$$
\e_n=\int_\bt |\p_n|^2\log |\p_n| d\s
$$
It is an easy exercise to show that $\e_n=\overline{o}(\sqrt n)$. In
this paper, we construct a measure from the Szeg\H{o} class for which this estimate
is sharp (over a
subsequence of $n$'s).
\end{abstract}

\maketitle

\section{Introduction}\label{s0}
Let $\s$ be a probability Borel measure on the unit circle $\bt=\{z: |z|=1\}$.
 The moments $c_k=c_k(\s)$, the Schur parameters $\g_k=\g_k(\s)$, the orthonormal
  polynomials $\p_n=\p_n(\s)$ with respect to the measure as well as their monic
  versions $\pp_n=\pp_n(\s)$ are defined in the standard way, see Simon \cite[Ch. 1]{si1}
  for definitions and terminology. We often indicate the dependence on the measure explicitly
   to avoid the misunderstanding.

It is quite reasonable to ask the following question: does some
additional condition on the measure provide  nontrivial bounds on
the size of the polynomials $\p_n$ beyond the normalization
\[
\int_{\mathbb{T}} |\phi_n(z)|^2d\sigma=1\, ?
\]
The size can be controlled by $L^p(d\sigma)$ norm ($p>2$) or by
other quantities. This problem is classical and was addressed,
for instance, in the framework of Steklov's conjecture \cite{ra1} by Rakhmanov (see also \cite{amb})
where the $L^\infty(\mathbb{T})$ norms were studied.

In this paper, we measure the size of the orthonormal polynomials by taking the entropy
integrals
\begin{equation}\label{e1}
\e_n(\s)=\int_\bt |\p_n|^2\log |\p_n| d\s
\end{equation}
Notice here that (set $\log x=\log^+ x -\log^- x$)
\[
\int_\bt |\p_n|^2\log^- |\p_n| d\s<1
\]
so only
\[
\e_n^+=\int_\bt |\p_n|^2\log^+ |\p_n| d\s
\]
can contribute to the growth of $\e_n$.

We say that $\s$ is a Szeg\H o measure (notation: $\s\in\sz$), if
its singular part $\s_s$ is arbitrary, and
$$
\int_\bt \log \s' dm>-\infty,
$$
where $\s'$ is the density of the absolutely continuous (a.c., for shorthand) part of $\s$ and
 $dm=dm(t)= d\th/(2\pi), t=e^{i\th}\in \bt, $ is the normalized
 Lebesgue measure on $\bt$. One might think that the Szeg\H{o} condition is relevant
 to the entropy integrals for the following reason. Assume first
 that $\sigma$ is purely absolutely continuous with the smooth positive
 density: $d\sigma=p(\theta)dm$ and
 $p(\theta)=|\pi(\theta)|^{-2}$ where $\pi(z)$ is an outer function on
 $\mathbb{D}$ such that $\pi^{-1}(z)$ is in the Hardy space $H^2(\mathbb{D})$. Then,
 one can easily show that $\phi_n(z)$ goes to $\pi(z)$
 uniformly on $\overline{\mathbb{D}}$. What happens to the entropy
 integrals? Obviously,
 \begin{equation}
 \e_n=\int_\mathbb{T} |\p_n|^2 \log|\p_n|\frac{dm}{|\pi(z)|^2}\to
 \int_\mathbb{T} \log |\pi(z)|dm \label{nons}
 \end{equation}
Now, it one considers $\sigma\in (S)$ instead, then the convergence
of the polynomials is not uniform but the right-hand side in
(\ref{nons}) does exist. So, one can conjecture that $\e_n$ has a
limit without any smoothness assumptions and that the only condition
needed is $\sigma\in (S)$. This conjecture is well-known in
the orthogonal polynomials community and attracted some attention
recently (see Beckermann et al. \cite{beck} and Aptekarev et al. \cite{ap1, ap2}). In \cite{beck}, for example, the entropies were studied
for the polynomials on the real line and under additional assumption
that the measure is a.c.

In this paper, we do not make this additional assumption. We conjecture that the construction from theorem 1.1 can be adjusted to produce an a.c.
measure $\sigma$ (see remark \ref{reser} below).

In the following theorem, we construct a Szeg\H o measure with
unbounded $\e_n$'s thus proving that
the above reasoning (\ref{nons}) is not true for general Szeg\H o measures.

\begin{theorem}\label{t1} There is $\s\in\sz$ and a subsequence $\{M_k\}$ such that
$$
\e_{M_k}(\s)=\int_\bt |\p_{M_k}(\s)|^2\log |\p_{M_k}(\s)| d\s =\bar
o(\sqrt{M_k})
$$
as $k\to\infty$.
\end{theorem}
The symbol $\bar o(\sqrt{M_k})$ means $h(M_k)\,\sqrt{M_k}$ with any
$h:\bn\to\br_+$ satisfying the property $\lim_{n\to\infty} h(n)=0$;
the point being that $h(n)$ can decay arbitrarily slowly. It follows
from the discussion in section \ref{s1} that this result is sharp and
the bound cannot be improved.

A by-product of the proof of this theorem is a result for the growth
of other integrals that measure the size of the polynomials,
 see corollary \ref{c1}. A simple counterpart of  theorem \ref{t1} also holds
 for orthogonal polynomials with respect to a Szeg\H o measure on an interval of
 the real line.

For $A_n, B_n>0$, we write $A_n \simeq B_n$ iff $c_1\le A_n/B_n\le
c_2$ with constants $c_1, c_2>0$. Similarly, $A_n\ggg B_n$ means
$A_n\ge c_1 B_n$. The symbol $\wto$ stands for the weak convergence
of measures.


\section{Preliminaries and Main lemma}\label{s1}
We begin with several simple observations:
\begin{itemize}
\item If $\|\gamma\|_\infty<1/2$, one has
\begin{equation}\label{e2}
\int_\bt \log\s' dm= \sum_k \log (1-|\g_k|^2) \simeq -\sum_k
|\g_k|^2
\end{equation}
where the both sides could be equal to $-\infty$. They are finite iff
$\s\in\sz$ (see \cite{sze}, \cite[p. 136, formula (2.3.1)]{si1}).
\item Let $\kappa_n$ be the leading coefficient of  $\p_n$. It is well-known that
\begin{equation}\label{e21}
\kappa^2_n= \frac1{\prod^{n-1}_{k=0}(1-|\g_k|^2)},
\end{equation}
so $\sup_n \kappa_n<\infty$ iff $\s\in\sz$. Hence, we can study the
entropy of monic polynomials instead, i.e.
\begin{equation}\label{e22}
\hat\e_{n}=\int_\bt |\pp_n|^2\log |\pp_n| d\s, \quad
\hat\e_{n}^+=\int_\bt |\pp_n|^2\log^+ |\pp_n| d\s
\end{equation}
We will do just that, the estimates obtained will imply theorem
\ref{t1}.

\item An upper bound for $\hat\e_n$ is easy to obtain. Recalling the Szeg\H o recurrence
formulas \cite[theorem 1.5.2]{si1} (notice that our $\g_n$ are $-\a_n$ from the book)
\begin{equation}
\left\{
\begin{array}{ccc}
\pp_{n+1}&=&z\pp_n +\bar\g_n \pp_n^*, \quad \pp_0=1, \label{e23}\\
\pp^*_{n+1}&=&\pp^*_n+\g_n z\pp_n, \quad \pp^*_0=1
\end{array}
\right.
\end{equation}
and  $|\pp_n(z)|=|\pp_n^*(z)|, \, z\in \mathbb{T}$ we see that
$$
|\pp_n(z)|\le \prod_{k=0}^{n-1}(1+|\g_k|)
$$
for $z\in \bt$,  and
$$
\log  |\pp_n(z)|\le \sum_{k=0}^{n-1}\log (1+|\g_k|)\le
\sum_{k=0}^{n-1} |\g_k|
$$
Since $\s\in\sz$, $\{\g_k\}\in \ell^2(\mathbb{Z}^+)$, and the latter
sum in the displayed formula is $\ov(\sqrt n)$. Hence
$$
\hat\e_n=\int_\bt |\pp_n|^2\log |\pp_n| d\s\le \lt(\sum_{k=0}^{n-1}
|\g_k|\rt)
 \int_\bt |\pp_n|^2 d\s = \ov(\sqrt n)
$$
\end{itemize}

Now we need to introduce some definitions to be used later in the
text.\smallskip

Let $\mu$ be a probability measure on $\bt$ with Schur parameters
$\{\g_k(\mu)\}$ and corresponding orthogonal polynomials
$\{\p_n(\mu)\}$. Given integers $N'<N$ and arbitrary $\kappa>0$, we
introduce the so-called $(N', N; \kappa)$--transformation of the
measure (or, equivalently, of its Schur parameters). Strictly
speaking, the $(N', N; \kappa)$--transformation depends also on
$\{\g'_k\}_{k=N'+1,\dots, N}$, a ``new interval" of Schur parameters
we want to ``incorporate" into $\{\g_k\}$. However, we will suppress
this dependence to keep the notation reasonably simple.

{\bf Definition of $(N',N;\kappa)$--transformation.} First, consider

\[
d\mu_0[\mu]=\frac{dm}{|\p_{N'+1}(\mu)|^2}
\]
 This measure is the so-called
Bernstein-Szeg\H{o} approximation to $d\mu$. Its Schur coefficients
 $\g_k(\mu_0)$ satisfy (\cite{sze, si1}):
$\g_k(\mu_0)=\g_k(\mu),
 \ k=0,\dots, N'$, $\g_k(\mu_0)=0, \  k> N'$. Secondly,
define the new sequence of Schur parameters by
\begin{equation}\label{e3}
\g_k(\mu_1)=\lt\{
\begin{array}{lll}
\g_k(\mu_0)&,& k=0,\dots, N'\\
\g'_k&,& k=N'+1,\dots N\\
0&,& k>N
\end{array}
\rt.
\end{equation}
This corresponds of course to writing
$d\mu_1[\mu]=1/|\p_{N+1}(\mu_1)|^2\, dm$ and the polynomial
$\p_{N+1}(d\mu_1)$ is determined through Schur parameters by (\ref{e21})
and (\ref{e23}). Next, we let
\begin{equation}\label{e4}
d\s[\mu]=\frac 1{1+\kappa}\lt(d\mu_1+ \kappa\, d\d_1\rt)
\end{equation}
where $\d_1$ is the Dirac's delta measure at $z=1$ on the unit
circle. The measure $\s$  and its  Schur coefficients $\{\g_k(\s)\}$
are  called the $(N', N; \kappa)$-transformation of $\mu$ and its
Schur coefficients $\{\g_k(\mu)\}$, respectively. Notice that the
normalization in (\ref{e4}) guarantees that $\sigma$ is a probability
measure.\bigskip

We define now the functions $\ga, \ps: \br^n_+\to \br_+ $ depending
on $\{\g_k\}_{k=1,\dots, n}$ as
\begin{eqnarray}\label{e401}
\ga_n&=&\ga_n(\{\g_k\})=\frac{\lt(\sum^n_{j=1}\g_j \rt)\
\exp\lt(\sum^n_{j=1}\g_j\rt)}{\sum^n_{j=1} \exp\lt(\sum^j_{k=1}\g_k\rt)}\\
\ps_n&=&\ps_n(\{\g_k\})=\frac{\exp\lt(\sum^n_{j=1}\g_j\rt)}{\sum^n_{j=1}
\exp\lt(\sum^j_{k=1}\g_k\rt)} \nonumber
\end{eqnarray}

\begin{lemma}\label{import} Given any small $L>0$, there exists an increasing sequence
$\{N_k\}\subset \bn$ and
$N_k$--tuples $\{\g_j\}_{j=1,\dots, N_k}, 0<\g_j<1,$ such that
$\sum_{j=1}^{N_k} \g_j^2\simeq L^2$ and
\begin{equation}\label{e402}
\ga_\nk(\{\g_j\})\ggg L^{4}\sqrt\nk, \qquad \ps_\nk(\{\g_j\})
\gtrsim L^{3}
\end{equation}
\end{lemma}

\begin{proof}
Rewrite $\ga_n$ as
$$
\ga_n=\ga_n(\{\g_t\})=\frac{\sum^n_{t=1}\g_t}{1+\sum^{n-1}_{m=1}
\exp\lt(-\sum^n_{t=m+1}\g_t\rt)}
$$
and make the change of summation  index $k\to n-k$,
$\hat\gamma_k=\gamma_{n-k}$. We have
\begin{equation}\label{e403}
\ga_n=\ga_n(\{\g_t\})=\frac{\sum^n_{t=1}\hat\g_t}{1+\sum^{n-1}_{m=1}
\exp\lt(-\sum^m_{t=1}\hat\g_t\rt)}
\end{equation}
Define the sequence $\{\nk\}$ by recursion
\[
N_{k+1}=N_k+\left[\frac{1}{k^2} \exp\left( L \sum_{j=1}^k
\frac{\sqrt {N_j-N_{j-1}}}{j^2}\right)\right]
\]
where $N_0=0$ and $[x]$ is the integer part of $x$. Taking
$N_1=CL^{-3}$ with $C$ large enough, one obtains by induction that
\begin{equation}
N_{k+1}-N_k\simeq N_{k+1} \label{one}
\end{equation}
for any $k$. Then, for each $k$, we choose the following $\{\hat\gamma_j\}$
\begin{equation}\label{e404}
\hat\gamma_t=\frac{L\, \beta_j}{\sqrt{N_j-N_{j-1}}},\quad t\in
(N_{j-1},N_j],
\end{equation}
$j=1,\dots, k$.  Furthermore, $\beta_j=j^{-2}, j=1,\dots, k-1$, and
$\beta_k=1$. We have
\[
\sum_{t=1}^{N_k} {\hat\gamma}_t^2\simeq L^2 \lt(\sum_{j=1}^{k-1}
\frac1{j^4}+1\rt) \simeq L^2
\]
and hence the first condition on the $N_k$--tuple $\{\g_j\}$ is
satisfied. Let us compute \eqref{e403} now. For the numerator,
\begin{equation}\label{l1n}
\sum_{t=1}^{N_k} \hat\gamma_t\simeq L \lt(\sum_{j=1}^{k-1}
\frac{\sqrt{N_j-N_{j-1}}}{j^2} +\sqrt{N_k-N_{k-1}}\rt) \simeq L
\sqrt{N_k}
\end{equation}
due to (\ref{one}).

 Next, estimating the denominator in (\ref{e403}), we have
\begin{eqnarray*}
&&\sum\limits_{m=1}^{N_k} \exp \left(-\sum\limits_{t=1}^m
\hat\gamma_t\right)\lesssim N_1+\sum_{j=2}^k (N_j-N_{j-1})
\exp\left(-\sum_{l=1}^{j-1}
(N_l-N_{l-1})\hat\gamma_l\right)\\
&&=N_1+\sum_{j=2}^k (N_j-N_{j-1}) \exp\left(- L \sum_{l=1}^{j-1}
\frac{\sqrt{N_l-N_{l-1}}}{l^2}\right)\\ &&\lesssim
L^{-3}+\sum_{j=2}^k \frac{1}{(j-1)^2}\lesssim L^{-3}
\end{eqnarray*}
by the definition of $\{N_k\}$. Combining the previous two bounds,
we get \eqref{e402}.
\end{proof}
The estimates obtained are not sharp in $L\sim 0$ at all.
However, they are sharp in $n$ and this is all we need.

\begin{remark}\label{r0} The estimate (\ref{l1n}) yields
\begin{equation}\label{e405}
\sum^{N_k}_{j=1} \exp\lt(\sum^j_{l=1}\g_l\rt) \ge
\exp\lt(\sum^{N_k}_{l=1}\g_l\rt) \ge \exp(CL\sqrt{N_k})
\end{equation}

\end{remark}

\begin{remark}\label{r1}
Let $\ss_L=\{\{\g_j\}_{j=1,\dots, \nk}: \sum_{j=1}^{\nk}
\g^2_j=L^2\}$
with a small $L>0$.Then,
$$
\max_{\{\g_j\}\subset \ss_L} \ga_\nk(\{\g_j\})\ggg L^{4}\sqrt \nk
$$
The bound
$$
\max_{\{\g_j\}\subset \ss_L} \ga_\nk(\{\g_j\})\lll L\sqrt \nk
$$
trivially follows from the definition and the Cauchy-Schwarz
inequality.
\end{remark}
\begin{remark}\label{r2}
The reasoning of the above lemma can be adapted to handle any
sufficiently large $n$ and not necessarily constructed as a sequence
$\{\nk\}$.
\end{remark}

 Recalling the definition of the
$(N',N;\kappa)$--transformation, we have the following key lemma.

\begin{lemma}\label{l2} Let $\mu\in \sz$ be a probability measure on
$\bt$ with real Schur parameters. For any natural $N'$, small
positive $L$, and $\d>0$, there is $\s$, a $(N',N;
\kappa)$--transformation of
 $\mu$ such that:
\begin{enumerate}
\item $N\ge 2 N'$,
\item $\sum^{N}_{k=N'+1} |\g'_k|^2\lll L^2$,
\item $0<\kappa<\d$,
\item Finally,
$$
\hat\e_{N}(\s)=\int_\bt |\pp_{N}(\s)|^2\log |\pp_{N}(\s)| d\s \ggg
L^4\sqrt{N}
$$
\end{enumerate}
\end{lemma}

\begin{proof} We start the proof with some simple observations. Let,
as above, $\mu_0=\mu_0[\mu]$, $\mu_1=\mu_1[\mu]$ and  $\s=\s[\mu]$.
First, assuming that such a transformation exists  and using
\eqref{e2}, we see that
$$
\int_\bt \log \mu'_0 dm= \sum_{k=0}^{N'} \log(1-|\g_k(\mu)|^2)
$$
and
\begin{eqnarray}
\sum_{k=0}^\infty \log(1-|\gamma_k(\sigma)|^2)= \int_\bt \log \s' dm=\int_\bt \log \mu'_1 dm-\log(\kappa+1)=\nonumber
\\
\sum^{N'}_{k=0} \log(1-|\g_k(\mu)|^2) +
\sum^{N}_{k=N'+1}\log(1-|\g'_k|^2)-\log(1+\kappa)\label{e51}
\end{eqnarray}
This estimate controls the  growth of the $\ell^2$--norm of Schur
coefficients
 under our transformation.

The right choice for the index $N$  will be made below; from now on,
we assume that it satisfies (1).  We define the Schur coefficients
$\{\g'_k\}_{k=N'+1,\dots, N}$  as
$$
\g'_{N+1-t}=\frac 12\, \hat\g_t, \quad t=1, \dots, N-N'
$$
where $\{\hat\g_t\}_{t=1,\dots, N-N'}$ comes from \eqref{e404}.
Notice that $\g'_k>0$ and  the sequence  $\{\g'_k\}$ satisfies (2).

Introduce the  Christoffel-Darboux kernel
\begin{equation}\label{e6}
K_n(\mu_1)(z,w)=\sum_{k=0}^n \p_k(\mu_1)(z)\ovl{\p_k(\mu_1)(w)}
\end{equation}
We define $\kappa$ as
\begin{equation}\label{e7}
\kappa=\frac 1{K_{N-1}(\mu_1)(1,1)}
\end{equation}
and we need a bound from below for $K_{N}(\mu_1)(1,1)$. Notice that
all Schur coefficients are real so $\pp_j(\mu_1)(1)$ are real and
$\pp_j(\mu_1)(1)=\pp_j^*(\mu_1)(1)$. For brevity, take
$A=|\pp_{N'+1}(\mu_1)(1)|$. All zeroes of $\Phi_j(\mu_1)(z)$ are
inside $\mathbb{D}$ so $A> 0$. Then, by Szeg\H o recurrence
relations \eqref{e23}, one has
$$
|\pp_{m}(\mu_1)(1)|=A\cdot \left|\prod^{m-1}_{k=N'+1}
(1+\g'_k)\right|, \quad m>N'+1
$$
Hence,
\begin{eqnarray*}
K_{N-1}(\mu_1)(1,1)&=&\sum_{k=0}^{N-1} |\p_k(\mu_1)(1)|^2\gtrsim
\sum^{N-1}_{k=N'+2}
|\pp_k(\mu_1)(1)|^2 \\
&=& A^2 \sum^{N-1}_{k=N'+2} \prod^{k-1}_{j=N'+1}(1+\g'_j)^2\simeq  A^2
 \sum^{N-1}_{l=N'+2} \exp \lt(2\sum^{l-1}_{j=N'+1}\g'_j\rt)
\end{eqnarray*}
By remark \ref{r0},  the latter quantity goes to infinity through a constructed subsequence in $N$, so,
recalling  \eqref{e7}, we obtain (3) for $N$ large enough.

 To start with (4), recall the following formula usually attributed to Geronimus
 (see, e.g., \cite[p. 253]{ra1} or \cite[p. 38, (3.30)]{gero};
 this very formula was used by Rakhmanov in his paper on the Steklov's
 conjecture~\cite{ra1})
 $$
 \pp_{N}(\s)(z)=\pp_{N}(\mu_1)(z)- \frac{\kappa \pp_{N}(\mu_1)(1)}
 {1+\kappa K_{N-1}(1,1)}\, K_{N-1}(\mu_1)(z,1)
 $$
 Consequently,
 $$
 \pp_{N}(\s)(1)= \pp_{N}(\mu_1)(1)/2
 $$
 and
 \begin{eqnarray}
 \hat\e_{N}(\s)&=&\int_\bt   |\pp_{N}(\s)|^2 \log |\pp_{N}(\s)| d\s  \nonumber\\
 &\ge&   \kappa|\pp_{N}(\s)(1)|^2 \log |\pp_{N}(\s)(1)|-C \,  \nonumber\\
 &\gtrsim& \frac{A^2 \prod^{N-1}_{k=N'+1} (1+\g'_k)^2\cdot \log\lt(A \prod^{N-1}_{k=N'+1}
 (1+\g'_k)/2\rt)}{\sum_{k=0}^{N'} |\pp_k(\mu_1)(1)|^2+ \sum^{N-1}_{k=N'+1} A^2 \prod^k_{j=N'+1}
 (1+\g'_j)^2} \nonumber\\
 &\gtrsim &  \frac{A^2\exp(\sum^{N}_{k=N'+1} 2\g'_k)\lt(\log A -\log 2+\sum^{N}_{k=N'+1} 2\g'_k\rt)}
 {\sum_{k=0}^{N'} |\pp_k(\mu_1)(1)|^2+A^2 \sum^{N-1}_{k=N'+1}\exp(\sum^k_{j=N'+1} 2\g'_j)}
  \nonumber\\
 && \label{e81}
\end{eqnarray}
Recalling \eqref{e402}, we continue as
$$
... \ggg L^4 \sqrt{N-N'}\ggg L^4\sqrt N
$$
whenever $N$ is large enough and belongs to the subsequence from
lemma~\ref{import}.
\end{proof}

\begin{remark}\label{reser}
The above lemma along with remark \ref{r2} imply the sharp
bound
\[
\sup_{\sigma:\,\{\|\gamma\|_2<1/2\}} \epsilon_n(\sigma)\simeq \sqrt{n}
\]
In our construction, the measure $\sigma$ yielding the lower bound
contained a jump. However, taking the Bernstein-Szeg\H{o}
approximations $\sigma_j$ to $\sigma$, we obtain
\[
\epsilon_n(\sigma_j)=\int_\mathbb{T} |\phi_n(\sigma_j)|^2\log
|\phi_n(\sigma_j)|d\sigma_j=\int_\mathbb{T} |\phi_n(\sigma)|^2\log
|\phi_n(\sigma)|d\sigma_j\to
\]
\[
\int_\mathbb{T} |\phi_n(\sigma)|^2\log |\phi_n(\sigma)|d\sigma
\]
since $\sigma_j\wto \sigma$. Thus, we have
\[
\sup_{\sigma:\,\{\|\gamma\|_2<1/2, \,\sigma_{s}=0\}}
\epsilon_n(\sigma)\simeq \sqrt{n}
\]
where $\sigma_s$ is the singular component of the measure $\sigma$.

\end{remark}

\section{Proof of  theorem \ref{t1} and some corollaries}\label{s2}

\nt{\it Proof of theorem \ref{t1}.}\quad Let $\delta_k>0$, $L_k>0$,
$\sum_{k=1}^\infty \delta_k<\infty$,
and $\sum_{k=1}^\infty L^2_k<\infty$. Assume that the both sums are small.
 The construction will recursively use lemma \ref{l2} from the previous section.
 We will construct the sequence of probability measure $\sigma_j$ by
 applying the $(M',M;\kappa)$--transformation consecutively
 (properly choosing parameters $M',M,\kappa$  at every step)
 and then will take
  the weak limit of $\{\sigma_j\}$. The measure $\sigma$ obtained
  in this way will have the necessary properties.

{\bf First step: $k=1$.} Let $d\mu^0=d\s^0=dm$, the Lebesgue
measure on $\bt$. Take $M'_1=1$; then, by lemma \ref{l2}, there is a
$(M'_1,M_1; \kappa_1)$--transformation of $\s^0$ which is denoted by
$\s^1$; it depends on the sequence of Schur parameters
$\{\g'_{1k}\}_{k=M'_1+1,\dots, M_1}$. We can arrange $M_1\ge 2^1$
and $M_1\gg 1/L^8_1$.  Also,
$$
\sum_{k=M'_1+1}^{M_1} |\g'_{1k}|^2\le L^2_1
$$
and, again by the same lemma
$$
\hat\e_{M_1}(\s^1)=\int_\bt |\pp_{M_1}(\s^1)|^2\log
|\pp_{M_1}(\s^1)| d\s^1 \ggg L_1^4\sqrt{M_1}
$$

Notice that for the Schur parameters of $\sigma^1$ we have
\begin{equation}\label{st1}
\sum_{l=0}^\infty \log(1-|\gamma_l(\sigma^1)|^2)=
\sum_{l=M_1'+1}^{M_1} \log(1-|\gamma_{1l}'|^2)-\log(1+\kappa_1)
\end{equation}
due to (\ref{e51}). We can choose $M_1$ such that
$\kappa_1<\delta_1$.

 \emph {For each measure $\sigma^k$ we construct later, introduce
\[T_{M_j}(\s^k)=|\pp_{M_j}(\s^k)|^2\log |\pp_{M_j}(\s^k)|, \ j\le k\]
This function is continuous on $\mathbb{T}$ since $\pp_l(\s^k)$ has all its zeroes
inside $\mathbb{D}$.}

{\bf Second step: $k=2$.}  Consider $T_{M_1}(\s^1)$; this is a
continuous function and hence there is a trigonometric polynomial
$f_1$ such that
\begin{equation}
||T_{M_1}(\s^1)-f_1||_\infty <\ve' \label{approx}
\end{equation}
 for any fixed $\ve'>0$. Let
$M'_2=\max\{\deg f_1, M_1\}$. Define $\s^2$ as $(M'_2, M_2;
\kappa_2)$--transformation of $\s^1$. Once again, we choose $M_2\ge
2^2$, $M_2\gg 1/L^8_2$ and
\[
\sum_{k=M'_2+1}^{M_2} |\g'_{2k}|^2\le L^2_2
\]
By the construction,

\begin{equation}\label{refe1}
\sum_{l=0}^\infty \log(1-|\gamma_l(\sigma^2)|^2)=\sum_{l=0}^{M_2'} \log(1-|\gamma_{l}(\sigma^1)|^2)+
\sum_{l=M_2'+1}^{M_2} \log(1-|\gamma_{2l}'|^2)-\log(1+\kappa_2)
\end{equation}
and, again, we take $\kappa_2<\delta_2$ to have $\|\gamma(\sigma^2)\|_2$ under control.

By lemma \ref{l2}
$$
\hat{\e}_{M_2}(\s^2)=\int_\bt |\pp_{M_2}(\s^2)|^2\log
|\pp_{M_2}(\s^2)| d\s^2 \ggg L_2^4\sqrt{M_2}
$$

Now, we could continue to apply the same procedure to generate
measures $\sigma_k$ with

\begin{equation}
\hat{\e}_{M_k}(\s^k)=\int_\bt |\pp_{M_k}(\s^k)|^2\log
|\pp_{M_k}(\s^k)| d\s^k \ggg
L_k^4\sqrt{M_k}=\overline{o}(\sqrt{M_k}) \label{more}
\end{equation}
where $\overline{o}(\sqrt{M_k})$ decays to zero arbitrarily slower
than $M_k$ simply because $L_k$ is fixed and $M_k$ can be chosen
large enough to accommodate arbitrarily slow decay. However, we  want more
than (\ref{more}): we want every measure $\sigma^k$ to have all of
the entropies $\{\hat{\epsilon}_{M_1}(\sigma^k), \ldots,
\hat{\epsilon}_{M_{k-1}}(\sigma^k)\}$ large. That, as we will see
next, can also be achieved by the choice of large $M_k$.

Let us handle the case $k=2$ first. We need to make
$\hat\epsilon_{M_1}(\sigma^2)$ large.  Recall the definition of
$(N',N;\kappa)$--transform. Set $\mu^2=\mu_0[\s^1]$ and
$\s^2=\s[\s^1]$. We have $\g_l(\s^1)=\g_l(\mu^2),
c_l(\s^1)=c_l(\mu^2)$ for $l\le M_2'$ where $c_l(\cdot)$ are the
moments of the measures. Therefore,
\[
\pp_{M_1}(\s^1)=\pp_{M_1}(\mu^2)\] \[T_{M_1}(\s^1)=T_{M_1}(\mu^2)\]
 Since $f_1$ is a trigonometric polynomial
of degree smaller than $M_2'$, we have

\[
\int_\mathbb{T} T_{M_1}(\sigma^1)d\sigma^1-\epsilon'\leq
\int_\mathbb{T} f_1d\sigma^1=\int_\mathbb{T} f_1d\mu^2\leq
\int_\mathbb{T} T_{M_1}(\mu^2)d\mu^2+\epsilon'
\]
by (\ref{approx}). Therefore,
\[
\int_\mathbb{T} T_{M_1}(\sigma^1)d\sigma^1-2\epsilon'\leq
\int_\mathbb{T} T_{M_1}(\mu^2)d\mu^2
\]
Now, notice that
\[
c_j(\sigma^2)\to c_j(\mu^2), \quad j\leq M_2'
\]
as $\kappa_2\to 0$ and this convergence is uniform in the choice of
$M_2$ and $\{\gamma_j(\sigma^2), j>M_2'\}$. Indeed, it follows from the representation
\[
c_j(\sigma^2)=\frac{c_j(\mu^2)+\kappa_2}{1+\kappa_2}, \quad j\leq M_2'
\]

Therefore
\[
\pp_{M_1}(\sigma^2)\to\pp_{M_1}(\mu^2)
\]
and, recalling that $\pp_{M_1}(\mu^2)$ has no zeroes on
$\mathbb{T}$,
\[
T_{M_1}(\sigma^2)=|\pp_{M_1}(\sigma^2)|^2\log
|\pp_{M_1}(\sigma^2)|\to |\pp_{M_1}(\mu^2)|^2\log
|\pp_{M_1}(\mu^2)|=T_{M_1}(\mu^2)
\]
uniformly on $\mathbb{T}$.
Thus,
\[
\int_\mathbb{T} T_{M_1} (\sigma^2)d\sigma^2\to \int_\mathbb{T}
T_{M_1}(\mu^2)d\s^2
\]
and
\[
\int_{\mathbb{T}} f_1d\sigma^2\to \int_{\mathbb{T}} f_1d\mu^2,\quad {\rm as}\quad \kappa_2\to 0
\]
again, uniformly in the choice of $M_2$ and $\{\gamma_j(\sigma^2),
j>M_2'\}$.
On the other hand, we always have
\[
\Bigl|\int_{\mathbb{T}} f_1d\mu^2-\int_{\mathbb{T}} T_{M_1}(\mu^2)d\mu^2\Bigr|\leq \epsilon'
\]
\[
\Bigl|\int_{\mathbb{T}} f_1d\sigma^2-\int_{\mathbb{T}} T_{M_1}(\mu^2)d\sigma^2\Bigr|\leq \epsilon'
\]
Thus, we only need to make sure that $\kappa_2$ is small enough to
guarantee
\[
\int_\mathbb{T} T_{M_1} (\sigma^2)d\sigma^2> \int_\mathbb{T}
T_{M_1}(\mu^2)d\mu^2-3\epsilon'
\]
and then
\[
\int_\mathbb{T} T_{M_1} (\sigma^2)d\sigma^2>\int_\mathbb{T}
T_{M_1}(\sigma^1)d\sigma^1-5\epsilon'>L_1^4\sqrt{M_1}-5\epsilon'
\]

{\bf $k$--th step.} Similarly, we construct the measure $\sigma^k$
such that (\ref{more}) holds along with
\begin{equation}
\int_\mathbb{T} T_{M_j} (\sigma^k)d\sigma^k>\int_\mathbb{T}
T_{M_j}(\sigma^j)d\sigma^j-5\epsilon'>L_j^4\sqrt{M_j}-5\epsilon', \,
j<k \label{b11}
\end{equation}
Moreover, we have
\begin{equation}
\sum_{l=0}^\infty \log(1-|\g_l(\s^k)|^2)\gtrsim -\sum_{j=1}^k L_j^2 -\sum_{j=1}^k \delta_j
\end{equation}
by induction (check, e.g., (\ref{st1}) and (\ref{refe1})).
From the construction it is clear that $\{\s^k\}$ converges weakly
to some $\sigma$. Indeed, at each step we have a recursion
\[
c_p(\sigma^{k+1})=\frac{c_p(\sigma^k)+\kappa_{k+1}}{1+\kappa_{k+1}}
\]
where $p$ is fixed. Since $\kappa_k<\delta_k$ and $\|\delta\|_1\ll 1$, we have convergence of $c_p(\sigma^k)$ by Cauchy criterion.
That is equivalent to $\sigma^k\wto \sigma$.

Repeating the arguments given above and using
(\ref{b11}), we obtain
\begin{eqnarray*}
\int_\bt |\pp_{M_j}(\s)|^2 \log |\pp_{M_j}(\s)| d\s&=&\lim_{k\to\infty} \int_\bt
|\pp_{M_j}(\s^k)|^2 \log |\pp_{M_j}(\s^k)| d\s^k\\
&\ge& \e_{M_j}(\s^j)-5\epsilon'\ggg
L_j^4\sqrt{M_j}=\overline{o}(\sqrt{M_j})
\end{eqnarray*}
for any fixed $j$. For the $\ell^2$--norm of the Schur parameters,
we have
\begin{equation}
\sum_{l=0}^\infty |\g_l(\s^k)|^2\lesssim \sum_{k=1}^\infty
L_k^2+\sum_{k=1}^\infty \delta_k
\end{equation}
The theorem is proved.
 \hfill $\Box$

One can obtain the following striking generalization.  Let $F:\br_+\to \br_+$ be
an increasing continuous function such that $\lim_{x\to\infty}
F(x)/x=\infty$. The
proof of the next statement repeats the arguments of the previous proof word for
word.

\begin{Corollary}\label{c1} For any gauge $F$, there is  $\s\in\sz$ such that
$$
\e_{M_k,F}(\s)=\int_\bt F(|\pp_{M_k}(\s)|^2) d\s\to\infty
$$
for some $M_k\to\infty$.
\end{Corollary}

As one can expect,  theorem \ref{t1} can be transferred to the
polynomials on an interval of the real line. We say that $\ti\s$ is a
Szeg\H o measure on $[-1,1]$, if it has an arbitrary  singular part
and
$$
\int^1_{-1} \frac{\log \ti\s'}{\sqrt{1-x^2}}\, dx>-\infty
$$
The orthonormal polynomials with respect to the measure $\ti\s$ are denoted $p_n=p_n(\ti\s)$.

\begin{Corollary}\label{c2}  There is a Szeg\H o measure $\ti\s$  on $[-1,1]$ and a
subsequence $\{M_k\}$ such that
$$
\e_{M_k}=\int^1_{-1} |p_{M_k}(\ti\s)|^2\log |p_{M_k}(\ti\s)| d\ti\s=\bar o(\sqrt{M_k})
$$
as $k\to\infty$.
\end{Corollary}

\nt{\it Sketch of the proof.} \ Notice that the measure $\s$ from
theorem \ref{t1} is symmetric on $\bt$. Consequently, its Schur
coefficients as well as coefficients of corresponding orthonormal
polynomials are real. Let $\ti\s$ denote the image of the measure
$\s$ constructed in theorem \ref{t1} from $\bt$ to $[-1,1]$ by the
(usual) transformation $x=\cos \th=(z+z^{-1})/2, \ z=e^{i\th}\in \bt$.

The classical formula \cite[theorem 11.5]{sze} implies that
$$
p_{n}(x)\simeq  {\bar z^n \p_{2n}(z)+ z^n
\p_{2n}(\bar z)}
$$
if $\sigma\in (S)$.
We adjust the construction of theorem \ref{t1} such that $M_k$
are all even. Then $|p_{M_k/2}(1)|\simeq |\p_{M_k}(1)|$.

Consequently,
\begin{eqnarray*}
\e^+_{M_k/2}(\ti\s)&=&\int^1_{-1} |p_{M_k/2}(\ti\s)|^2\log^+
|p_{M_k/2}(\ti\s)| d\ti\s\\
&\ggg&  \epsilon^+_{M_k}(\sigma^k)\gtrsim \bar o(\sqrt{M_k})
\end{eqnarray*}
by theorem \ref{t1} as the value at $z=1$ alone provides the necessary growth of the entropy. \hfill$\Box$\bigskip

It is an interesting question to find a natural class of
measures for which the polynomial entropy integrals are bounded.
It is likely that by improving the technique of \cite{ra1, amb}
one can show that the Steklov's condition on the measure is not good
enough for the entropies to be uniformly bounded. In the meantime,
it is quite possible that fairly mild conditions are sufficient for the averages of $\epsilon_n$ to be under control.

\medskip
\nt{\it Acknowledgments.} The authors would like to thank A.~Aptekarev and A.~Mart\' inez-Finkelshtein for drawing their attention
to the problem and for numerous stimulating discussions. The first
author acknowledges the hospitality of the Institute for Advanced
Study and the Institute of Mathematics of Bordeaux, University of
Bordeaux 1 where  part of this research was done.

\end{document}